\documentclass[review]{elsarticle}

\usepackage{lineno,hyperref}
\usepackage{amsmath}
\usepackage{amsfonts}
\newtheorem{theorem}{Theorem}
\newtheorem{definition}{Definition}
\newtheorem{property}{Property}
\newtheorem{remark}{Remark}
\newtheorem{proof}{Proof}
\newtheorem{lemma}{Lemma}

\begin{document}

\begin{frontmatter}

\title{Raising the regularity of generalized Abel equations in fractional Sobolev spaces with homogeneous boundary conditions}

\author{Yulong Li\fnref{myfootnote}}
\address{Science, Mathematics and Technology  Cluster,
	Singapore University of Technology and Design,
	8 Somapah Road, 487372, Singapore}


\ead{liyulong0807101@gmail.com}


\begin{abstract}
The generalized (or coupled) Abel equations on the bounded interval have been well investigated in H$\ddot{\text{o}}$lderian  spaces that admit integrable singularities at the endpoints and relatively inadequate in other functional spaces. In recent years, such operators have appeared in BVPs of fractional-order differential equations such as fractional diffusion equations that are usually studied in the frame of fractional Sobolev spaces for weak solution and numerical approximation; and their analysis  plays the key role during the process of converting weak solutions to the true solutions.  

 This article develops the mapping properties of generalized Abel operators $\alpha {_aD_x^{-s}}+\beta {_xD_b^{-s}}$ in fractional Sobolev spaces, where $0<\alpha,\beta$, $\alpha+\beta=1$, $ 0<s<1$ and $ {_aD_x^{-s}}$, $ {_xD_b^{-s}}$ are fractional Riemann-Liouville integrals. It is mainly concerned with the regularity property of $(\alpha {_aD_x^{-s}}+\beta {_xD_b^{-s}})u=f$ by taking into account homogeneous boundary conditions. Namely, we investigate the regularity behavior of $u(x)$ while letting $f(x)$ become smoother and imposing homogeneous boundary restrictions $u(a)=u(b)=0$.
\end{abstract}


\end{frontmatter}


\section{Background and motivation}\label{sec:introduction}
The equation
\begin{equation}\label{equ:abelequations}
\alpha(x) \int_a^x \frac{u(t)dt}{(x-t)^{\mu}}+\beta(x) \int_x^b \frac{u(t)dt}{(t-x)^{\mu}}=f(x), a<x<b,
\end{equation}
with $0<\mu<1$, is known as the generalized Abel equation with exterior coefficients on the interval, which has appeared in various fractional-order differential equations (see e.g. model \eqref{equ:fractionaldiffusion} below). Various theoretical aspects of problem  \eqref{equ:abelequations} and its generalizations to  the curve in complex plane have been investigated by many mathematicians. A thorough historical record of results on the topic of \eqref{equ:abelequations} are collected in chapter 6, 7 of monograph \cite{MR1347689}, the first survey article in \cite{almeida2013advances} and, especially, the references and remarks listed in there. For a general treatment of problem  \eqref{equ:abelequations} in connection with the Hilbert boundary value problem, readers can refer to \cite{gakhov2014boundary}, \cite{muskhelishvili2008singular}.

Among those results, in this work, we mainly concern the solvability and mapping properties  in fractional Sobolev spaces for $\alpha(x), \beta(x)$ being constants, i.e. 

\begin{equation}\label{equ:CAbelEequa}
Au:=\alpha \int_a^x \frac{u(t)dt}{(x-t)^{\mu}}+\beta \int_x^b \frac{u(t)dt}{(t-x)^{\mu}}=f(x), a<x<b.
\end{equation}

$\eqref{equ:CAbelEequa}$ can be solved for $u(x)$ in different closed  forms.  Wolfersdorf in \cite{wolfersdorf1969losung} (1969) provided an expression of $u(x)$ in terms of the kernal involving hypergeometric function and other similar but different forms appeared in later papers after this work. In 1978, Samko provided another interesting method in his PhD thesis \cite{samkothesis} (Russian) which opens a wider door (also presented in English version book \cite{MR1347689}, section 30). He constructed the explicit expression of solution $u(x)$ by using singular integral operators (see Def. \ref{def:integral operator}) and  sought the solutions in the H$\ddot{\text{o}}$lderian  space that admits integrable singularities at the boundary, namely $H^*(\Omega)$ (see sequel notations in \ref{Notations}). And in this setting it follows that the operator $A$ maps the space $H^*(\Omega)$ one-to-one and onto  a better space $H^*_{1-\mu}(\Omega)$ consisting of all functions of the form ${_aD_x^{-(1-\mu)}}\psi$ (or symmetrically ${_xD_b^{-(1-\mu)}}\psi'$) with $\psi$ (or $\psi'$) $\in H^*(\Omega)$. These fundamental results are summarized in Theorem 13.14 and Theorem 30.7 in \cite{MR1347689} (also presented in sequel Section \ref{sec:FRLO}) and will be the basement of this work.

However, from a practical point of view, these results are still inadequate for the  purpose of the analysis and application of fractional-order differential equations involving the generalized Abel operator $A$. The reason is illustrated as follows.

In the modern theory of integer-order elliptic equations, the usual strategy for studying the true solutions is two steps: first establish the weak solution in Sobolev spaces and then raise the regularity of  weak solution to recover the classic solution. It is not surprising that this idea can be naturally extended to the study of fractional-order elliptic equations, which consists of a large class of important models in physics and other scientific areas (good references for applications of fractional elliptic and general fractional differential equations can be referred to  multi-volume Handbook of Fractional Calculus with Applications:  Vol 1 - Vol 8, and books such as \cite{kilbas2006theory}\cite{sabatier2007advances}\cite{tarasov2011fractional}\cite{kulish2002application} etc. and recent articles in journals such as FCAA, PFDA, JFCA and EJDE). Let us take the following boundary value problem of  fractional diffusion equation as an example (this model is proposed and discussed in Ervin and his collaborators' series work \cite{ernmpde06}\cite{MR3802435}\cite{zheng2020wellposedness}\cite{zheng2020optimal} and  a generalization to higher dimension in \cite{ervin2007variational}):
\begin{equation}\label{equ:fractionaldiffusion}
\begin{cases}
[L(u)](x)= f(x),\, x\in \Omega=(a,b),\\
u(a)=u(b)=0,\\
[L(u)](x) := -Dk(x)(\alpha\, {_aD_x^{-(1-\mu)}}+\beta\,{_xD_b^{-(1-\mu)}})Du\\
\qquad \qquad+p(x)Du+q(x)u(x),\\
0<\alpha, \beta <1, \alpha+ \beta=1, 0<\mu<1.
\end{cases}
\end{equation}
By imposing suitable conditions on the coefficients and $f(x)$, one can first establish the existence of weak solution of \eqref{equ:fractionaldiffusion} in fractional Sobolev spaces and secondly raise the regularity of weak solution  to recover the true solution. The first stage is analogous to the classic case once constructing some intermediate functional spaces by using Riemann-Liouville derivatives to connect with usual fractional Sobolev spaces, and has been resolved in \cite{ernmpde06} (2006); the second stage, however, is challenging and the most recent results on this topic are presented in published work \cite{MR3802435} (2018), dissertation \cite{MR4024334} (2019) (chapter 6) and  manuscripts \cite{ervin2019regularity} (2019), \cite{li2020skewed} (2020). The challenging during the phase of raising regularity is attributed to two aspects: 
\begin{enumerate}
	\item 
	Raising the regularity of weak solution of \eqref{equ:fractionaldiffusion} essentially amounts to raising the regularity of $u(x)$ in problem \eqref{equ:CAbelEequa} in fractional Sobolev spaces (since the weak solution of fractional differential equations are sought in fractional Sobolev spaces), instead of the usual space $H^*(\Omega)$ where \eqref{equ:CAbelEequa} has been well studied.
	\item Meanwhile, we need to take into account the homogeneous boundary restriction, namely $u(a)=u(b)=0$ in problem \eqref{equ:CAbelEequa}.
\end{enumerate}

Roughly speaking,  the restoration of true solution of model \eqref{equ:fractionaldiffusion} is intertwined with problem \eqref{equ:CAbelEequa} and can be casted into the scenario: Assume $u$ belongs to the fractional Sobolev space $\widehat{H}^{(1+\mu)/2}_0(\Omega)$ and satisfies 
\begin{equation}\label{equ:thelastone}
\alpha \int_a^x \frac{u(t)dt}{(x-t)^{\mu}}+\beta \int_x^b \frac{u(t)dt}{(t-x)^{\mu}}=f(x), a<x<b.
\end{equation}
Let right-hand side $f(x)$ in \eqref{equ:thelastone} become smoother and smoother, one needs to show that $u(x)$ becomes smoother as well (this corresponds to the phase of the weak solution of \eqref{equ:fractionaldiffusion} recovering to the true solution).   Rephrasing in a little bit more precise way, let $f(x)$ is of the Riemann-Liouville fractional integral form ${_aD_x^{-s_1}}\psi_1$ with $\psi_1\in H^*(\Omega), s_1>0$, one needs to prove that the solution $u(x)$ has also a representation ${_aD_x^{-s_2}}\psi_2$ with another $\psi_2\in H^*(\Omega), s_2>0$ and that  $s_2$ increases as  $s_1$ increases.

This is the motivation and requirement of the  work and will be successfully answered in this paper.

Since for $0<s_1<1$, the consequence is well known and is implied by Theorem 13.14  in \cite{MR1347689}, we are going to develop the results only for the case $1\leq s_1$. Note for $1\leq s_1$,   $f(x)={_aD_x^{-s_1}}\psi_1$ can actually be equivalently rewritten as $f(x)={_xD_b^{-s_1}}\psi_1' + c$  for another $\psi_1' \in H^*(\Omega)$ and a certain constant $c$, which suggests that there is an intrinsic constant on the right-hand side \eqref{equ:thelastone} and usually $c$ is non-zero. (That is why we formulate the specific form in right-hand side of \eqref{equ:firstone} below, and moreover, that form is necessary and sufficient for the application of fractional diffusion equations).

Our analysis is mainly based on Samko's results that obtained in space $ H^*(\Omega)$ and we intend to connect them to  the fractional Sobolev space $\widehat{H}^{(1+\mu)/2}_0(\Omega)$. The main theorem that will be proved is  the following.

\begin{theorem}\label{theoremmain}
	Given  $0<\sigma<1$,  let $c$ be a constant, $\psi(x)\in \widehat{H}^{(1+\mu)/2}_0(\Omega)$ and $f(x)\in H(\overline{\Omega}) \cap H^*_\sigma(\Omega)$. If 
	\begin{equation}\label{equ:firstone}
	\alpha _aD_x^{-(1-\mu)}\psi+\beta {_xD_b^{-(1-\mu)}}\psi \overset{a.e.}{=}{_aD_x^{-1}}f +c, x\in\Omega,
	\end{equation}
	then the solution $\psi(x)$ has a representation
	\begin{equation}
	\psi(x)={_aD_x^{-(\mu+\sigma)}}K_\sigma,
	\end{equation}
	where function $K_\sigma(x)$ belongs to $H^*(\Omega)$ ($K_\sigma$ depends on $\sigma$).
\end{theorem}

It is worth noting that the assumption on $\psi$ and $f$ made in Theorem~\ref{theoremmain} are usually automatically satisfied for free, during applications of  Theorem~\ref{theoremmain} to fractional differential equations. Indeed, $\psi(x)\in \widehat{H}^{(1+\mu)/2}_0(\Omega)$ comes from the existence of weak solution of the corresponding differential equations and $f(x)\in H(\overline{\Omega}) \cap H^*_\sigma(\Omega)$ will be  implied by the smoothness of coefficients. 	Therefore, the results of this work have direct applications in fractional-order elliptic equations and provide  essential techniques in the process of raising the regularity of weak solutions to recover classic solutions. A more detailed application and relationship between \eqref{equ:fractionaldiffusion} and \eqref{equ:thelastone} can be found in manuscript \cite{li2020skewed}.

The material is organized as follows.
\begin{itemize}
	\item  Section \ref{Notations} is about convention and notation.
	\item   Section \ref{sec:FRLO} gathers and lists necessary definitions and associated properties, most of which have been given specific citation information for readers' convenience. 
	\item Main results are established in Section \ref{sec:raisingtheregularity}, including two lemmas and one theorem.
\end{itemize}

\section{Convention and notation}\label{Notations}

Convention
\begin{itemize}
	\item    $\Omega=(a,b), \overline{\Omega}=[a,b]$ and $-\infty <a,b<\infty$, whenever they appear throughout the material.
	\item We shall often not distinguish $``="$ at every point from $``="$ almost every point  in those equations when there is no chance of misunderstanding. On some occasions, the  notation $\overset{a.e.}{=}$  will be used  for  emphasizing the validity for ``almost everywhere"  to draw the reader's attention for those cases that are of importance.
	\item  Whenever we deal with a function $f(x)$ belonging to Sobolev or $L^p$ spaces, it is implicitly assumed that $f$ denotes a suitable representative of the equivalence classes, unless otherwise specified.
	\item All the functions are default to be real-valued, and all the constants that will appear in different contexts will be assumed to be real constants.
\end{itemize}

Notation
\begin{itemize}
	\item $H^\lambda(\overline{\Omega})$: H$\ddot{\text{o}}$lderian  space ($\lambda>0$).
	\vspace{0.2cm}
	\item $H(\overline{\Omega}):=\underset{\lambda>0}{\bigcup}H^\lambda(\overline{\Omega})$.
	\vspace{0.2cm}
	\item $H^\lambda_0(\overline{\Omega}):=\{f:f(x)\in H^\lambda(\overline{\Omega}), f(a)=f(b)=0\}$.
	\vspace{0.2cm}
	\item $H^\lambda(\rho, \overline{\Omega})$: weighted H$\ddot{\text{o}}$lderian  space, ($\rho(x)=\prod_{k=1}^{n}|x-x_k|^{\mu_k}, x_k, x\in \overline{\Omega} $, $n$ is a positive integer).
	\vspace{0.2cm}
	\item  $H^\lambda_0(\rho, \overline{\Omega}):=\{f: f(x)\in H^\lambda(\rho, \overline{\Omega}), \rho(x_k) f(x_k)=0, k=1,\cdots, n\}$.
	\vspace{0.2 cm}
	\item  
	$
	H_0^\lambda(\epsilon_1,\epsilon_2):=\{f:f(x)=\frac{g(x)}{(x-a)^{1-\epsilon_1}(b-x)^{1-\epsilon_2}}, g(x)\in H^\lambda_0(\overline{\Omega})\}.
	$
	\vspace{0.1cm}
	\item $
	H^*(\Omega):=\underset{\underset{0<\lambda\leq 1}{0<\epsilon_1,\epsilon_2}}{\cup} H_0^\lambda(\epsilon_1,\epsilon_2).
	$
	\vspace{0.2cm}
	\item  $
	H^*_\sigma(\Omega):=\underset{\underset{\sigma<\lambda\leq 1}{0<\epsilon_1,\epsilon_2}}{\cup} H_0^\lambda(\epsilon_1,\epsilon_2).
	$
	\vspace{0.2cm}
	\item ${_a}D_{x}^{-\sigma}$, ${_x}D_{b}^{-\sigma}$, $\boldsymbol{D}^{-\sigma}$ and $\boldsymbol{D}^{-\sigma * }$ represent fractional integrals if $\sigma>0$, identity operators if $\sigma=0$ and fractional derivatives if $\sigma<0$ (see specific definitions in Section \ref{sec:FRLO}).
	\vspace{0.1cm}
	\item $AC(\overline{\Omega})$: the set of absolutely continuous functions on $\overline{\Omega}$.
	\item  $C(G)$: the set of all continuous functions on a  set $G$.
	\item $C^n(G):=\{f: f^{(n)}(x)\in C(G)\}$.
	\item Both $Df$ and $f'$ represent the usual derivative of a function $f$.
	\item $(f , g)_\Omega$ and $(f ,g )_\mathbb{R}$ denote the  integrals $\int_\Omega fg$ and $\int_\mathbb{R} fg$, respectively.
	\item $C_0^\infty(\Omega)$ consists of all the infinitely  differentiable functions on $\Omega$ and with compact support in $\Omega$.
\end{itemize}

\section{Prerequisite Knowledge}\label{sec:FRLO}
Necessary preliminaries that will be of use are presented first, including definitions and associated properties. These definitions and properties are known and most of them are quoted from the literature directly. 
Properties given in this section may not be in the strongest forms, but they are adequate for our purpose.

%
\subsection{Riemann-Liouville integrals and their properties}
\label{ssec:FRLI}
\begin{definition} \label{def:RLI}
	Let $w:(c,d) \rightarrow \mathbb{R}, (c,d) \subset \mathbb{R}$ and  $\sigma >0$. The left and right Riemann-Liouville fractional integrals of order $\sigma$ are, formally respectively, defined as
	\begin{align}
		({_a}D_{x}^{-\sigma}w)(x) &:= \dfrac{1}{\Gamma(\sigma)}\int_{a}^{x}(x-s)^{\sigma -1}w(s) \, {\rm d}s, \label{eq:LRLI}\\
		({_x}D_{b}^{-\sigma} w)(x) &:= \dfrac{1}{\Gamma(\sigma)}\int_{x}^{b}(s-x)^{\sigma-1}w(s) \, {\rm d}s,  \label{eq:RRLI}
	\end{align}
	where $\Gamma(\sigma)$ is  Gamma function. For convenience, when $c=-\infty$ or $d=\infty$ we set
	\begin{equation}\label{eq:LRRLI}
	(\boldsymbol{D}^{-\sigma}w)(x) :={_{-\infty}}D_{x}^{-\sigma} w \text{ and }
	(\boldsymbol{D}^{-\sigma * }w)(x) :={_{x}}D_{\infty}^{-\sigma} w. 
	\end{equation}
\end{definition}
In particular, if $\sigma=0$, ${_a}D_{x}^{-\sigma}$, ${_x}D_{b}^{-\sigma}$, $\boldsymbol{D}^{-\sigma}$ and $\boldsymbol{D}^{-\sigma * }$ are regarded as identity operators. 
\begin{property}[\cite{MR1347689}, eq. (2.19), p. 34]\label{p-reflection}
	Let $\sigma>0$ and $(Qf)(x)=f(a+b-x)$, then the following operators are reflective:
	\begin{equation}
	Q{_aD_x^{-\sigma}}Q={_xD_b^{-\sigma}}.
	\end{equation}
\end{property}
\begin{property}\label{pro-mappings}
	If $0<\sigma<1$, $1<p<1/\sigma$, then the fractional operators ${_aD_x^{-\sigma}}$, ${_xD_b^{-\sigma}}$ are bounded from $L^p(\Omega)$ into $L^q(\Omega)$ with $q=\frac{p}{1-\sigma p}$.
\end{property}
\begin{property}\label{pro:MappingIntoHolderianFromLp}
	If $0<\frac{1}{p}<\sigma<1+\frac{1}{p}$, fractional operators ${_aD_x^{-\sigma}}$ and ${_xD^{-\sigma}_b}$  map the space $L^p(\Omega)$ into the H$\ddot{\text{o}}$lderian  space $H^{\sigma-1/p}(\overline{\Omega})$.
\end{property}
\begin{remark}
	Property~\ref{pro-mappings} is a combination of Theorem 3.5 (\cite{MR1347689}, p. 66) and Property~\ref{p-reflection},  and Property~\ref{pro:MappingIntoHolderianFromLp} is a combination of Corollary of Theorem 3.6 (\cite{MR1347689}, p. 69) and Property~\ref{p-reflection}.
\end{remark}
\subsection{Riemann-Liouville derivatives and their properties}
\begin{definition}\label{def:RLD}
	Let $w:(c,d) \rightarrow \mathbb{R}, (c,d) \subset \mathbb{R}$ and  $\sigma >0$. Assume  $n$ is the smallest integer greater than $\sigma$ (i.e., $n-1 \leq \sigma < n$). The left and right Riemann-Liouville fractional derivatives of order $\sigma$ are, formally respectively, defined as
	\[
	({_a}D_x^{\sigma}w)(x) := \frac{{\rm d}^n}{{\rm d}x^n} {_a}D_x^{\sigma-n} w  \text{ and }~
	({_x}D_b^{\sigma}w)(x) := (-1)^n \frac{{\rm d}^n}{{\rm d} x^n} {_x}D_b^{\sigma-n} w.
	\]
	For ease of notation, when $c=-\infty$ or $d=\infty$ we set
	\begin{equation}
	(\boldsymbol{D}^{\sigma} w )(x)= {_{-\infty}}D_{x}^{\mu}w \text{ and }
	(\boldsymbol{D}^{\sigma*}w)(x) =  {_{x}}D^{\mu}_{\infty}w .\label{def:Infty}
	\end{equation}
\end{definition}
\subsection{Functional spaces $H^*(\Omega)$, $H^*_\sigma(\Omega)$ and  properties}\label{Sec-IFS}
We list several mapping properties related to $H^*(\Omega)$ and $H^*_\sigma(\Omega)$, which play an important role in connecting the whole analysis of  this work.
\begin{property}[\cite{MR1347689}, Lemma 30.2, p. 618]\label{pro-WeightedMapping}
	Let $0<\sigma<1$, the weighted singular operator
	\begin{equation}
	(S_{\nu_a,\nu_b} f)(x)=\frac{1}{\pi}\int_a^b\left(\frac{x-a}{t-a}\right)^{\nu_a}\left(\frac{b-x}{b-t}\right)^{\nu_b}\frac{f(t)}{t-x}dt
	\end{equation}
	maps the space $H^*_\sigma(\Omega)$ into itself provided that $\sigma-1<\nu_a\leq  \sigma$ and $\sigma-1<\nu_b\leq \sigma$.
\end{property}
\begin{property}[\cite{MR1347689}, Theorem 13.14, p. 248]\label{pro-correspondence-1}
	Let $0<\sigma<1$, the fractional integration operators ${_aD_x^{-\sigma}}$ and ${_xD_b^{-\sigma}}$ map the space $H^*(\Omega)$ one-to-one and onto the space $H^*_\sigma(\Omega)$, respectively. Consequently, ${_aD_x^{-\sigma}}(H^*(\Omega))={_xD_b^{-\sigma}}(H^*(\Omega))$.
\end{property}
\begin{property}\label{Existence-Abel}
	Let $0<\sigma<1$, $\gamma_1, \gamma_2>0$, $\gamma_1{_aD_x^{-\sigma}}+\gamma_2{_xD_b^{-\sigma}}$ maps $H^*(\Omega)$ one-to-one and onto the space $H^*_\sigma(\Omega)$.
\end{property}
\begin{remark}
	Property~\ref{Existence-Abel} is a combination of  Theorem 30.7 (\cite{MR1347689}, p. 626) and Property \ref{pro-correspondence-1}.
\end{remark}
\subsection{Dominant singular integral and properties}
\begin{definition}\label{def:integral operator}
	The singular integral operator $S$ is formally defined as
	\[
	(S\psi)(x)=\frac{1}{\pi}\int_a^b \frac{\psi (t)}{t-x} dt, x\in \Omega,
	\]
\end{definition}
the convergence being understood in the principal value sense.
\begin{property}[\cite{MR1347689}, Corollary 2, p. 208]\label{pro:CR}
	Denote $r_a(x)=x-a$, $ r_b(x)=b-x, x\in \overline{\Omega}$.
	\[
	{_xD_b^{-\lambda}}(r_b^{-\lambda}S(r_b^\lambda\psi))=r_a^\lambda S(r_a^{-\lambda}{_xD_b^{-\lambda}}\psi)
	\]
	is valid for $0<\lambda<1, \psi \in L^p(\Omega), p> 1$.
\end{property}
\begin{property} [\cite{MR1347689}, Theorem 11.1, p. 200]  \label{pro:BInW}
	Let $n$ be a positive integer, $0<\lambda<1$, then the operator $S$ is bounded in the space $H^\lambda_0(\overline{\Omega})$, and in the weighted space $H^\lambda_0(\rho, \overline{\Omega})$,  $\rho(x)=\prod_{k=1}^{n}|x-x_k|^{\mu_k}, x_k, x\in \overline{\Omega} $, provided that $\lambda<\mu_k<\lambda+1  \,(k=1,\cdots, n)$.
\end{property}
\begin{property}\label{pro:SolvabilityofDSI}
	Let $c_1, c_2$ be constants and $c_1^2+c_2^2\neq 0$. Denote 
	$\frac{c_1-ic_2}{c_1+ic_2}=e^ {i\theta}$ and choose the value of $\theta$ so that $0\leq \theta<2\pi$.  Further denote spaces $X_1=H^*(\Omega)$, $X_2=H^*(\Omega)\cap C( (a,b] )$, $X_3=H^*(\Omega)\cap C([a,b))$  and $X_4=H^*(\Omega)\cap C([a,b])$, and
	define $n_a$ and $n_b$ as follows:
	\[
	\begin{aligned}
	n_a(X_1)&= n_a(X_2)=1, n_a(X_3)=n_a(X_4)=0;\\ 
	n_b(X_1)&=n_b(X_3)=1, n_b(X_2)=n_b(X_4)=0.
	\end{aligned}
	\]
	Consider  the problem  
	\begin{equation}\label{equ:TheoremDSI}
	\begin{cases}
	c_1\psi(x) + \frac{c_2}{\pi}\int_a^b\frac{\psi(t)}{t-x} dt=f(x), x\in \Omega,\\
	\text{where}\, f(x)=\frac{f_*(x)}{(x-a)^{1-\nu_a}(b-x)^{1-\nu_b}}, \, f_*(x)\in H(\overline{\Omega}), \nu_a,\nu_b\in \mathbb{R}.
	\end{cases}
	\end{equation}
	Then each of  the following holds:
	\begin{enumerate}
		\item  If \[
		1-n_a(X_i)-\frac{\theta}{2\pi}<\nu_a,\, \frac{\theta}{2\pi}-n_b(X_i)<\nu_b,
		\]
		then
		\eqref{equ:TheoremDSI} is  unconditionally solvable in  $X_i \,(i=1, 2, \text{or}\, 3)$ 
		and its general solution $\psi_i(x)$ in $X_i \,(i=1, 2, \text{or}\, 3)$ is given by
		\begin{equation}\label{eq-sloution-CSE}
		\begin{aligned}
		\psi_i(x)&=C(x-a)^{1-n_a(X_i)-\frac{\theta}{2\pi}}(b-x)^{\frac{\theta}{2\pi}-n_b(X_i)}+\frac{c_1f(x)}{c_1^2+c_2^2}-\\
		&\frac{c_2}{\pi(c_1^2+c_2^2)}\int_a^b\left( \frac{x-a}{t-a}\right)^{1-n_a(X_i)-\frac{\theta}{2\pi}}\left(\frac{b-x}{b-t}\right)^{\frac{\theta}{2\pi}-n_b(X_i)} \frac{f(t)}{t-x} dt ,
		\end{aligned}
		\end{equation}
		where 
		\[
		C=0 \,\, \text{for}\, i=2, 3, \, \text{and C is an arbitrary constant for}\,i=1.
		\]
		\item If \eqref{equ:TheoremDSI} is solvable in  $X_4$, then the solution $\psi_4(x)$ is unique and is given by
		\begin{equation}
		\begin{aligned}
		\psi_4(x)&=\frac{c_1f(x)}{c_1^2+c_2^2}-\\
		&\frac{c_2}{\pi(c_1^2+c_2^2)}\int_a^b\left( \frac{x-a}{t-a}\right)^{1-n_a(X_4)-\frac{\theta}{2\pi}}\left(\frac{b-x}{b-t}\right)^{\frac{\theta}{2\pi}-n_b(X_4)} \frac{f(t)}{t-x} dt .
		\end{aligned}
		\end{equation}
	\end{enumerate}
\end{property}
\begin{remark}
	According to the statement of Property \ref{pro:SolvabilityofDSI}, it is worth noting that if a solution of \eqref{equ:TheoremDSI} belongs to the space $X_4$, then it also lies in $X_i$ $(i=1,2,3)$ and therefore it has four equivalent representations which differ only in form. 
	This property is  a special case of Theorem 30.2 on page 609 in \cite{MR1347689} by letting $Z_0(x)=1$, $a_1(x)=c_1$ and $a_2(x)=c_2$ in our case. And in part $(2)$ we  omitted the sufficient and necessary condition for  \eqref{equ:TheoremDSI} to be solvable in  $X_4$ since we shall not need it.  
\end{remark}

\subsection{Facts of fractional Sobolev spaces }
It is known that  there are various ways to define fractional Sobolev spaces, which are essentially equivalent but serve as convenient tools for deriving various properties under different contexts. On the whole real axis, one way is in terms of the Fourier transform as follows.

\begin{definition}\label{thm:FTHsR}
	Given $0\leq s$, let
	\begin{equation}
	\widehat{H}^s(\mathbb{R}) := \left \{ w \in L^2(\mathbb{R}) : \int_{\mathbb{R}} (1 + |2\pi\xi|^{2s}) |\widehat{w}(\xi) |^2 \, {\rm d} \xi < \infty \right \}.
	\end{equation}

	It is endowed with semi-norm and norm
	\[
	|w|_{\widehat{H}^s(\mathbb{R})}:=\|(2\pi\xi)^s \widehat{w}\|_{L^2(\mathbb{R})},
	\|w\|_{\widehat{H}^s (\mathbb{R})}:=\left(\|w\|^2_{L^2(\mathbb{R})} +|w|^2_{\widehat{H}^s(\mathbb{R})}\right)^{1/2}.
	\]
\end{definition}
By restricting to the  bounded interval, we have 

\begin{definition}\label{def:ForInterval}
	Given $0\leq s$.
	\begin{equation}
	\widehat{H}^s_0(\Omega):=\{\text{Closure of}\, \, u\in C_0^\infty(\Omega) \, \text{with respect to norm}\, \|\tilde{u}\|_{\widehat{H}^s(\mathbb{R})} \},
	\end{equation}
	where notation $\tilde{u}$ denotes the extension of $u(x)$ by $0$ outside $\Omega$. It is endowed with semi-norm and norm
	\[
	|u|_{\widehat{H}^s_0(\Omega)}:=|\tilde{u}|_{\widehat{H}^s(\mathbb{R})},
	\|u\|_{\widehat{H}^s_0(\Omega)}:=\|\tilde{u}\|_{\widehat{H}^s(\mathbb{R})}.
	\]
\end{definition}
It is well-known that $\widehat{H}^s(\mathbb{R})$ is a Hilbert space and so is $\widehat{H}^s_0(\Omega)$.
%

%
Another equivalent definition is achieved with the aid of left or right fractional-order weak derivative, which is a generalization of integer-order weak derivative:
\begin{definition}(\cite{glsima18}, Section 3)\label{Equivalent-definition}
	Given $0\leq s$ and assume $u(x)\in L^2(\mathbb{R})$, then  $u(x) \in \widehat{H}^s(\mathbb{R})$ if and only if there exists a unique $\psi_1(x) \in L^2(\mathbb{R})$  such that 
	\begin{equation}
	\int_\mathbb{R} u \cdot\boldsymbol{D}^s \psi =\int_\mathbb{R} \psi_1 \cdot \psi\quad 
	\end{equation}
	for any $\psi\in C^\infty_0(\mathbb{R})$.\\
	Similarly, \\
	$u(x) \in \widehat{H}^s(\mathbb{R})$ if and only if there exists a unique $\psi_2(x) \in L^2(\mathbb{R})$  such that 
	\begin{equation}
	\int_\mathbb{R} u \cdot\boldsymbol{D}^{s*} \psi =\int_\mathbb{R} \psi_2 \cdot \psi
	\end{equation}
	for any $\psi\in C^\infty_0(\mathbb{R})$.
\end{definition}
Using this definition, it is easy to see
\begin{property}\label{pro:alternate-form}
	Given $\frac{1}{2}<s<1$, then $u\in \widehat{H}^s_0(\Omega)$ can be represented as
	\begin{equation}
	u(x)={_aD_x^{-s}\psi_1}={_xD_b^{-s}}\psi_2,
	\end{equation}
	for certain   $\psi_1$, $\psi_2 \in L^2(\Omega)$. As a consequence, ${_aD_x^s}u$ and ${_xD_b^s}u$ exist a.e. and coincide with $\psi_1$, $\psi_2$, respectively.
\end{property}

\section{Main results}\label{sec:raisingtheregularity}

Two lemmas that are of crucial toward the main theorem are proved first.

Let $0<\mu, \alpha, \beta <1, \alpha+\beta=1$ throughout this section (Section \ref{sec:raisingtheregularity}).

\begin{lemma}\label{lemma-Pre-RaiseRegularity}
	Let $A=\alpha-\beta\cos(\mu\pi)$, $B=\beta\sin(\mu\pi)$ and $f(x)\in H(\overline{\Omega})$. Denote  $r_b(x)=b-x, x\in \overline{\Omega}$, $F(x)=(b-x)^\mu f(x)$ and 
	$\frac{A-iB}{A+iB}=e^ {i\theta}$ with the value of $\theta$ chosen so that $0\leq \theta<2\pi$. 
	Consider the problem
	\begin{equation}\label{equ-RProblem-1}
	A\psi(x)+ \frac{B}{\pi}\int_a^b\frac{\psi(t)}{t-x} dt =F(x), x\in \Omega.
	\end{equation}
	Then each of the following is valid:
	\begin{enumerate}
		\item \eqref{equ-RProblem-1} is solvable in spaces $X_2=H^*(\Omega)\cap C( (a,b] )$ and $X_3=H^*(\Omega)\cap C([a,b))$ respectively, and its according solution $\psi_i   (i=2,3)$ is unique and is represented as
		\begin{equation}
		\begin{aligned}
		\psi_i(x)&=\frac{AF(x)}{A^2+B^2}-\\
		&\frac{B}{\pi(A^2+B^2)}\int_a^b\left( \frac{x-a}{t-a}\right)^{1-n_a(X_i)-\frac{\theta}{2\pi}}\left(\frac{b-x}{b-t}\right)^{\frac{\theta}{2\pi}-n_b(X_i)} \frac{F(t)}{t-x} dt ,
		\end{aligned}
		\end{equation}
		where
		\[
		\begin{aligned}
		n_a(X_2)&=1, n_a(X_3)=0;\\ 
		n_b(X_2)&=0, n_b(X_3)=1.
		\end{aligned}
		\]
		\item $\theta$ satisfies 
		\begin{equation}\label{equ:thetacondition}
		\mu<\frac{\theta}{2\pi}<1.
		\end{equation}
		\item   The solution $\psi_i(x)$ in part $(1)$ satisfies
		\begin{equation}
		\frac{\psi_i(x)}{(b-x)^\mu}\in H^*(\Omega), \, ( i=2,3).
		\end{equation}
		\item  The solution $\psi_2(x)$ in part $(1)$ satisfies
		\begin{equation}
		{_aD_x^{-\mu}} \frac{\psi_2}{r_b^\mu} \in L^{p}(I_b),
		\end{equation}	
		where $p=\frac{1}{1-\frac{\theta}{2\pi}},\,I_b=(\frac{b-a}{2}, b) $.
		
		\item  If the solution $\psi_i(x)$ in part $(1)$ satisfies $\psi_2(x)=\psi_3(x)$, $x\in \Omega$, then $\psi_2$ (or $\psi_3$, which is the same) has four equivalent representations:
		\begin{equation}
		\begin{aligned}
		\psi_2(x)&=\frac{AF(x)}{A^2+B^2}-\frac{B}{\pi(A^2+B^2)}\int_a^b\left( \frac{x-a}{t-a}\right)^{1-\frac{\theta}{2\pi}}\left(\frac{b-x}{b-t}\right)^{\frac{\theta}{2\pi}} \frac{F(t)}{t-x} dt ,\\
		&=\frac{AF(x)}{A^2+B^2}-\frac{B}{\pi(A^2+B^2)}\int_a^b\left( \frac{x-a}{t-a}\right)^{1-\frac{\theta}{2\pi}}\left(\frac{b-x}{b-t}\right)^{\frac{\theta}{2\pi}-1} \frac{F(t)}{t-x} dt ,\\
		&=\frac{AF(x)}{A^2+B^2}-\frac{B}{\pi(A^2+B^2)}\int_a^b\left( \frac{x-a}{t-a}\right)^{-\frac{\theta}{2\pi}}\left(\frac{b-x}{b-t}\right)^{\frac{\theta}{2\pi}} \frac{F(t)}{t-x} dt ,\\
		&=\frac{AF(x)}{A^2+B^2}-\frac{B}{\pi(A^2+B^2)}\int_a^b\left( \frac{x-a}{t-a}\right)^{-\frac{\theta}{2\pi}}\left(\frac{b-x}{b-t}\right)^{\frac{\theta}{2\pi}-1} \frac{F(t)}{t-x} dt .
		\end{aligned}
		\end{equation}
		
	\end{enumerate}
\end{lemma}
\begin{proof}

	1. Proof for part (1). 
	
	Let us see that   part (1) is just a direct consequence of  the first part of  Property~\ref{pro:SolvabilityofDSI}, we only need to justify the applicability of  Property~\ref{pro:SolvabilityofDSI}.
	
	We shall need to check three aspects.
	
	Firstly,  $A^2+B^2\neq 0$ by recalling $0<\mu, \alpha, \beta <1, \alpha+\beta=1$ (see the beginning of  Section \ref{sec:raisingtheregularity}).
	
	Secondly, the function $F(x)$ can be equivalently rewritten as
	\begin{equation}
	F(x)=(b-x)^\mu f(x)=\frac{(b-x)^{\mu} f(x)}{(x-a)^{1-1}(b-x)^{1-1}}.
	\end{equation}

	In the numerator,  since $f\in H(\overline{\Omega})$, there exists a $0<\lambda_0<\mu$ so that
	\begin{equation}\label{equ:lambda}
	f\in H^{\lambda_0}(\overline{\Omega}) ,
	\end{equation}
	and we directly check that 
	\[
	(b-x)^\mu \in H^{\mu}(\overline{\Omega}) 
	\]
	with  the assistance of the well-known inequality
	\begin{equation}\label{equ:wellknowninequality}
	\frac{|y_1^y-y_2^y|}{|y_1-y_2|^y}\leq 1,\, (0\leq y\leq 1, 0<y_1, 0<y_2, y_1\neq y_2).
	\end{equation}
	Therefore, their product $(b-x)^\mu f(x)\in H^{\lambda_0}(\overline{\Omega})$ follows.
	
	Lastly,   observe that $B\neq 0$,  which implies that  $\theta\neq 0$. Hence, 
	\[
	1-n_a(X_i)-\frac{\theta}{2\pi}<1, \frac{\theta}{2\pi}-n_b(X_i)<1
	\]
	hold for  $i=2,  3$. 
	
	So, all the hypotheses are met for applying  the first part of  Property~\ref{pro:SolvabilityofDSI},   the part $(1)$ of Lemma \ref{lemma-Pre-RaiseRegularity} follows.
	
	\vspace{0.2cm}
	
	2. Now we  prove the part (2).
	
	Since
	\begin{equation}\label{equ-AB-Theta}
	\frac{A-iB}{A+iB}=e^{i\theta}\, \text{and}\, 0\leq \theta<2\pi,
	\end{equation}
	it is clear that
	\[\frac{\theta}{2\pi}<1.
	\]
	
	Substituting for $A, B$ and simplifying \eqref{equ-AB-Theta}, we have
	\[
	\frac{\alpha-\beta\cos(\mu\pi)-i\beta\sin(\mu\pi)}{\alpha-\beta\cos(\mu\pi)+i\beta\sin(\mu\pi)}=\frac{\alpha-\beta e^{i\mu \pi}}{\alpha-\beta e^{-i\mu \pi}}=\frac{\frac{\alpha}{\beta}-e^{i\mu\pi}}{\frac{\alpha}{\beta}-e^{-i\mu\pi}}=e^{i\theta}.
	\]
	Solving the last equality  for $\frac{\alpha}{\beta}$ gives
	\[
	\frac{\alpha}{\beta}=\frac{e^{i(\theta-2\mu\pi)}-1}{e^{i\theta}-1}e^{i\mu\pi}.
	\]
	Taking the fact $e^{iz}-1=2ie^{iz/2}\sin(z/2), z\in \mathbb{C}$ into account, we arrive at
	\[
	\frac{\alpha}{\beta}=\frac{\sin((\theta-2\mu\pi)/2)}{\sin(\theta/2)}.
	\]
	Again, recalling $0<\mu, \alpha, \beta <1, \alpha+\beta=1$   we derive
	
	\[
	0<(\theta-2\mu\pi)/2<\pi.
	\]
	Hence, we see the assertion
	\[
	\mu<\frac{\theta}{2\pi}<1.
	\]
	
	This fact will be  used in the rest of  proof and  in the proof of  subsequent lemmas.
	
	\vspace{0.2cm}
	
	3. To prove the part (3), namely $\frac{\psi_i(x)}{(b-x)^\mu}\in H^*(\Omega) \, ( i=2,3)$, we discuss the two cases  separately in this step and the next step.
	
	For $i=2$, substituting for $\psi_2$ into $\frac{\psi_2(x)}{(b-x)^\mu}$ by using the representation in   part (1) of the lemma  and simplifying, we obtain
	\begin{equation}\label{equ-ForPsiTwoExpression}
	\begin{aligned}
	\frac{\psi_2(x)}{(b-x)^\mu}&=\frac{Af(x)}{A^2+B^2}-\frac{B}{\pi(A^2+B^2)}\int_a^b\left( \frac{x-a}{t-a}\right)^{-\frac{\theta}{2\pi}}\left(\frac{b-x}{b-t}\right)^{\frac{\theta}{2\pi}-\mu} \frac{f(t)}{t-x} dt \\
	&=\frac{A}{A^2+B^2}\Pi_1-\frac{B}{\pi(A^2+B^2)}\Pi_2.
	\end{aligned}
	\end{equation}
	
	Let us look at $\Pi_1$ and $\Pi_2$.
	
	$\Pi_1$, namely $f(x)$ (recall $f\in H^{\lambda_0}(\overline{\Omega})$ in \eqref{equ:lambda}), is relatively easily seen to belong to $H^*(\Omega)$ by  manipulating as follows:
	\begin{equation}\label{equ:fractionalpi1}
	\Pi_1=\frac{(x-a)^\epsilon f(x)(b-x)^\epsilon}{(x-a)^{1-(1-\epsilon)}(b-x)^{1-(1-\epsilon)}}, 
	\end{equation}
	where  $\epsilon$ is chosen so that $ 0<\epsilon<\lambda_0$. 
	
	In the numerator,  it can be directly verified with \eqref{equ:wellknowninequality} that
	\begin{equation}
	(x-a)^\epsilon,\,(b-x)^\epsilon\in H^\epsilon(\overline{\Omega}).
	\end{equation}

	Therefore, the product 
	\begin{equation}\label{equ-BelongtoHolderianSpace}
	(x-a)^\epsilon f(x)(b-x)^\epsilon\in H_0^\epsilon(\overline{\Omega})
	\end{equation}
	by taking into account the boundary and $f\in H^{\lambda_0}(\overline{\Omega})$. 
	
	In the denominator of \eqref{equ:fractionalpi1}, simply observe that $1-\epsilon>0$.
	
	Hence, $\Pi_1\in H^*(\Omega)$ follows by the definition of $H^*(\Omega)$ (see notation in Section~\ref{Notations}).
	
	Similarly, to see $\Pi_2\in H^*(\Omega)$, we rewrite
	\begin{equation}\label{equ-UsedForLp}
	\begin{aligned}
	\Pi_2&=\int_a^b\left( \frac{x-a}{t-a}\right)^{-\frac{\theta}{2\pi}}\left(\frac{b-x}{b-t}\right)^{\frac{\theta}{2\pi}-\mu} \frac{f(t)}{t-x} dt \\
	&=\frac{\int_a^b\left(\frac{x-a}{t-a}\right)^{\epsilon}\left(\frac{b-x}{b-t}\right)^{\theta/(2\pi)-\mu+\epsilon}\frac{(t-a)^{\theta/(2\pi)+\epsilon}f(t)(b-t)^\epsilon}{t-x} dt}{(x-a)^{1-(1-\theta/(2\pi)-\epsilon)}(b-x)^{1-(1-\epsilon)}},
	\end{aligned}
	\end{equation}
	where $\epsilon$ is chosen so that $0<\epsilon<\text{min}\{\lambda_0, 1-\frac{\theta}{2\pi}\}$ (it should be clear this $\epsilon$ is different  from  the one in \eqref{equ:fractionalpi1} and we will use nation $\epsilon$ this way several times in the rest of proof).
	
	In the denominator, $1-\theta/(2\pi)-\epsilon>0$, $1-\epsilon>0$.  
	
	According to the definition of   $H^*(\Omega)$, $\Pi_2$ belongs to $ H^*(\Omega)$ is ensured provided that the numerator is in $H^{\epsilon/2}_0(\overline{\Omega})$, namely
	\begin{equation}\label{equ-BelongtoHolderianSpace-1}
	\int_a^b\left(\frac{x-a}{t-a}\right)^{\epsilon}\left(\frac{b-x}{b-t}\right)^{\theta/(2\pi)-\mu+\epsilon}\frac{(t-a)^{\theta/(2\pi)+\epsilon}f(t)(b-t)^\epsilon}{t-x} dt\in H^{\epsilon/2}_0(\overline{\Omega}).
	\end{equation}
	
	\eqref{equ-BelongtoHolderianSpace-1} is verified by a direct application of  Property~\ref{pro:BInW} by checking that 
	\[
	(t-a)^{\theta/(2\pi)+\epsilon}f(t)(b-t)^\epsilon\in H^{\epsilon/2}_0(\overline{\Omega}),
	\]
	which can be justified analogously to \eqref{equ-BelongtoHolderianSpace}, \\
	and that  
	\[
	\epsilon/2<\epsilon<1+\epsilon/2, \quad \epsilon/2<\theta/(2\pi)-\mu+\epsilon<1+\epsilon/2,
	\]
	where the fact that $\frac{\theta}{2\pi}-\mu>0$ from  \eqref{equ:thetacondition} was used in the second piece.
	
	Hence, 
	\begin{equation}\label{equ:Pi2}
	\Pi_2\in H^*(\Omega),
	\end{equation} 
	and therefore, 
	\[
	\frac{\psi_2(x)}{(b-x)^\mu} \in H^*(\Omega).
	\]	 
	
	\vspace{0.2cm}
	
	4.  We continue to consider the case $i=3$.
	
	Substituting for $\psi_3(x)$ into $	\frac{\psi_3(x)}{(b-x)^\mu}$  by using the representation  in   part $(1)$ of the lemma  and simplifying, we have
	\begin{equation}
	\begin{aligned}
	\frac{\psi_3(x)}{(b-x)^\mu}&=\frac{Af(x)}{A^2+B^2}-\frac{B}{\pi(A^2+B^2)}\int_a^b\left( \frac{x-a}{t-a}\right)^{1-\frac{\theta}{2\pi}}\left(\frac{b-x}{b-t}\right)^{\frac{\theta}{2\pi}-1-\mu} \frac{f(t)}{t-x} dt \\
	&=\frac{A}{A^2+B^2}\Sigma_1-\frac{B}{\pi(A^2+B^2)}\Sigma_2.
	\end{aligned}
	\end{equation}
	
	For the first term, $\Sigma_1=f(x)$, which is the same as $\Pi_1$ in the previous step, thus, $\Sigma_1\in H^*(\Omega)$.
	
	For $\Sigma_2$,
	\begin{equation}
	\begin{aligned}
	\Sigma_2&=\int_a^b\left( \frac{x-a}{t-a}\right)^{1-\frac{\theta}{2\pi}}\left(\frac{b-x}{b-t}\right)^{\frac{\theta}{2\pi}-1-\mu} \frac{f(t)}{t-x} dt \\
	&=\frac{\int_a^b\left(\frac{x-a}{t-a}\right)^{1-\theta/(2\pi)+\epsilon}\left(\frac{b-x}{b-t}\right)^{\epsilon}\frac{(t-a)^{\epsilon}f(t)(b-t)^{1-(\theta/(2\pi)-\mu-\epsilon)}}{t-x} dt}{(x-a)^{1-(1-\epsilon)}(b-x)^{1-(\theta/(2\pi)-\mu-\epsilon)}},
	\end{aligned}
	\end{equation}
	where $\epsilon$ is chosen so that $0<\epsilon<\text{min} \{ \lambda_0, \frac{\theta}{2\pi}-\mu \}$.
	
	Again, observe that  in the denominator $1-\epsilon>0$  and $\frac{\theta}{2\pi}-\mu-\epsilon>0$. 
	
	Then by the definition of  $H^*(\Omega)$, $\Sigma_2$ belongs to $ H^*(\Omega)$ is guaranteed provided that the numerator is in $H^{\epsilon/2}_0(\overline{\Omega})$, namely
	\begin{equation}\label{equ-BelongtoHolderianSpace-2}
	\int_a^b\left(\frac{x-a}{t-a}\right)^{1-\theta/(2\pi)+\epsilon}\left(\frac{b-x}{b-t}\right)^{\epsilon}\frac{(t-a)^{\epsilon}f(t)(b-t)^{1-(\theta/(2\pi)-\mu-\epsilon)}}{t-x} dt \in H^{\epsilon/2}_0(\overline{\Omega}).
	\end{equation}
	
	\eqref{equ-BelongtoHolderianSpace-2} is guaranteed by  Property~\ref{pro:BInW} and can be justified similarly to \eqref{equ-BelongtoHolderianSpace-1} without essential difference.
	
	Hence, $\Sigma_2\in H^*(\Omega)$, and therefore, 
	\[
	\frac{\psi_3(x)}{(b-x)^\mu} \in H^*(\Omega).
	\]
	
	This completes the proof for part $(3)$.
	
	\vspace{0.2cm}
	
	5. Proof for part $(4)$.
	
	Using \eqref{equ-ForPsiTwoExpression} and integrating both sides by ${_aD_x^{-\mu}}$,
	\begin{equation}
	{_aD_x^{-\mu}} \frac{\psi_2}{r_b^\mu}=\frac{A}{A^2+B^2}	{_aD_x^{-\mu}} \Pi_1-\frac{B}{\pi(A^2+B^2)}	{_aD_x^{-\mu}} \Pi_2.
	\end{equation}
	
	It is clear that $	{_aD_x^{-\mu}}\Pi_1  \in L^{p}(I_b)$, $p=\frac{1}{1-\frac{\theta}{2\pi}}$, $I_b=(\frac{b-a}{2}, b) $,  since $\Pi_1 =f(x)\in H^{\lambda_0}(\overline{\Omega})$. In order to show  ${_aD_x^{-\mu}} \frac{\psi_2}{r_b^\mu} \in L^{p}(I_b)$, we only need to show ${_aD_x^{-\mu}} \Pi_2\in L^{p}(I_b)$.
	
	Recall from \eqref{equ:Pi2} that $\Pi_2\in H^*(\Omega)$, which  implies that $\Pi_2\in L^z(\Omega)$ for some $z>1$. This allows us to apply the fact (eq.  (11.17), p. 206, \cite{MR1347689}) that 
	\[
	{_aD_x^{-\mu}} g= \cos(\mu \pi){_xD_b^{-\mu}}g-\sin(\mu \pi) {_xD_b^{-\mu}}(r_b^{-\mu}S(r_b^{\mu}g)),\,\text{for}\, g(x)\in L^p(\Omega),p> 1
	\]
	to ${_aD_x^{-\mu}}\Pi_2$ to obtain
	\begin{equation}
	{_aD_x^{-\mu}}\Pi_2=\cos(\mu \pi){_xD_b^{-\mu}}\Pi_2-\sin(\mu \pi) {_xD_b^{-\mu}}(r_b^{-\mu}S(r_b^\mu \Pi_2)).
	\end{equation}
	
	Let us examine ${_xD_b^{-\mu}}\Pi_2$ first. Indeed,
	\begin{equation}
	\Pi_2\in L^{t}(I_b), \quad \text{for any}\quad t>0,
	\end{equation}
	by seeing that in \eqref{equ-UsedForLp} the $\epsilon$ can be chosen as small as possible.
	Certainly, 
	\begin{equation}\label{equ-111}
	{_xD_b^{-\mu}}\Pi_2\in L^p(I_b), \quad p=\frac{1}{1-\frac{\theta}{2\pi}}.
	\end{equation}
	
	Secondly, we look ${_xD_b^{-\mu}}(r_b^{-\mu}S(r_b^\mu \Pi_2))$.  Utilizing  the second line of  \eqref{equ-UsedForLp}, we obtain
	\begin{equation}\label{equ-Analogoue}
	r_b^\mu \Pi_2=\frac{\int_a^b\left(\frac{x-a}{t-a}\right)^{\epsilon}\left(\frac{b-x}{b-t}\right)^{\theta/(2\pi)+\epsilon}\frac{(t-a)^{\theta/(2\pi)+\epsilon}f(t)(b-t)^{\epsilon+\mu}}{t-x} dt}{(x-a)^{\theta/(2\pi)+\epsilon}(b-x)^{\epsilon}},
	\end{equation}
	where  $0<\epsilon<\text{min}\{\lambda_0, 1-\frac{\theta}{2\pi}\}$.
	
	On the right-hand side of ~\eqref{equ-Analogoue},  the numerator belongs to $H^{\epsilon/2}_0(\overline{\Omega})$, which is justified analogously to ~\eqref{equ-BelongtoHolderianSpace-1}. Therefore, by virtue of Property~\ref{pro:BInW},
	\begin{equation}\label{equ:Lemma1Pre}
	r_b^{-\mu} S(r_b^\mu \Pi_2)=\frac{h(x)}{(x-a)^{\theta/(2\pi)+\epsilon}(b-x)^{\epsilon+\mu}},
	\end{equation}
	for a certain $h(x)\in H^{\epsilon/2}_0(\overline{\Omega})$ and $\epsilon$ is the same   as in~\eqref{equ-Analogoue}. It is straightforward to check
	\[
	r_b^{-\mu} S(r_b^\mu \Pi_2)\in L^{1/(\mu+2\epsilon)}(I_b).
	\]

	Remembering that the $\epsilon$ can be chosen as small as possible in \eqref{equ:Lemma1Pre} and by another use of Property~\ref{pro-mappings} to \eqref{equ:Lemma1Pre} over the interval $(\frac{b-a}{2}, \, b)$, we  derive
	\begin{equation}
	{_xD_b^{-\mu}}(r_b^{-\mu}S(r_b^\mu \Pi_2))\in L^{\nu}(I_b),\quad \text{for any}\quad \nu >0.
	\end{equation}
	In particular,
	\begin{equation}\label{equ-222}
	{_xD_b^{-\mu}}(r_b^{-\mu}S(r_b^\mu \Pi_2))\in L^p(I_b), \quad  p=\frac{1}{1-\frac{\theta}{2\pi}}.
	\end{equation}
	Combining ~\eqref{equ-111} and ~\eqref{equ-222} gives the desired result
	\begin{equation}
	{_aD_x^{-\mu}} \frac{\psi_2}{r_b^\mu} \in L^{p}(I_b),
	\end{equation}	
	where 
	\[
	p=\frac{1}{1-\frac{\theta}{2\pi}},\,I_b=(\frac{b-a}{2}, b).
	\]

	6. Proof for part $(5)$.
	
	Since we assume the solutions satisfy $\psi_2(x)=\psi_3(x)$, $\psi_2$ (or $\psi_3$)  belongs to $H^*(\Omega)\cap C([a,b])$ due to $\psi_2\in H^*(\Omega)\cap C( (a,b] )$ and $\psi_3\in H^*(\Omega)\cap C([a,b))$. This means the problem  \eqref{equ-RProblem-1} is solvable in the following four spaces:
	\begin{equation}
	H^*(\Omega), H^*(\Omega)\cap C( (a,b] ),  H^*(\Omega)\cap C([a,b)), H^*(\Omega)\cap C([a,b]).
	\end{equation}
	
	From  Property~\ref{pro:SolvabilityofDSI}, we know that, as a solution of \eqref{equ-RProblem-1}, $\psi_2$ totally has four representations in these four spaces,  namely:
	\begin{equation}
	\begin{aligned}
	\psi_2(x)&=\frac{AF(x)}{A^2+B^2}-\frac{B}{\pi(A^2+B^2)}\int_a^b\left( \frac{x-a}{t-a}\right)^{1-\frac{\theta}{2\pi}}\left(\frac{b-x}{b-t}\right)^{\frac{\theta}{2\pi}} \frac{F(t)}{t-x} dt\\
	&\in H^*(\Omega)\cap C([a,b]),\\
	\psi_2(x)&=\frac{AF(x)}{A^2+B^2}-\frac{B}{\pi(A^2+B^2)}\int_a^b\left( \frac{x-a}{t-a}\right)^{1-\frac{\theta}{2\pi}}\left(\frac{b-x}{b-t}\right)^{\frac{\theta}{2\pi}-1} \frac{F(t)}{t-x} dt\\
	&\in H^*(\Omega)\cap C( [a,b) ),\\
	\psi_2(x)&=\frac{AF(x)}{A^2+B^2}-\frac{B}{\pi(A^2+B^2)}\int_a^b\left( \frac{x-a}{t-a}\right)^{-\frac{\theta}{2\pi}}\left(\frac{b-x}{b-t}\right)^{\frac{\theta}{2\pi}} \frac{F(t)}{t-x} dt \\
	&\in H^*(\Omega)\cap C((a,b]),\\
	\psi_2(x)&=\frac{AF(x)}{A^2+B^2}-\frac{B}{\pi(A^2+B^2)}\int_a^b\left( \frac{x-a}{t-a}\right)^{-\frac{\theta}{2\pi}}\left(\frac{b-x}{b-t}\right)^{\frac{\theta}{2\pi}-1} \frac{F(t)}{t-x} dt\\
	&+C\, (x-a)^{-\frac{\theta}{2\pi}}(b-x)^{\frac{\theta}{2\pi}-1}\\
	&\in H^*(\Omega).
	\end{aligned}
	\end{equation}
	
	This completes the proof  provided that $C=0$ in the forth equation. 
	
	To see this, first we calculate
	\begin{equation}
	\begin{aligned}
	C&=\left( \psi_2(x)-\frac{AF(x)}{A^2+B^2}\right) (x-a)^{\frac{\theta}{2\pi}}(b-x)^{1-\frac{\theta}{2\pi}}\\
	&+\frac{B}{\pi(A^2+B^2)}\int_a^b\frac{(t-a)^{\theta/(2\pi)}f(t)(b-t)^{1-(\theta/(2\pi)-\mu)}}{t-x} dt,\, x\in \Omega.
	\end{aligned}
	\end{equation}
	
	The first term equals to $0$ at the boundary point $x=b$.
	
	In the second term, we can directly verify, analogously to \eqref{equ-BelongtoHolderianSpace}, that
	\[
	(t-a)^{\theta/(2\pi)}f(t)(b-t)^{1-(\theta/(2\pi)-\mu)}\in H^{\epsilon_0}_0(\overline{\Omega}), \epsilon_0=\text{min}\{\lambda_0,\frac{\theta}{2\pi}, 1-(\frac{\theta}{2\pi}-\mu)\}.
	\]
	Thus, in light of  Property \ref{pro:BInW}, we see
	\[
	\frac{1}{\pi}\int_a^b\frac{(t-a)^{\theta/(2\pi)}f(t)(b-t)^{1-(\theta/(2\pi)-\mu)}}{t-x} dt \in H^{\epsilon_0}_0(\overline{\Omega}).
	\]
	
	Consequently, 
	\begin{equation}
	\begin{aligned}
	C&=\lim_{x\rightarrow b^-} \left( \psi_2(x)-\frac{AF(x)}{A^2+B^2}\right) (x-a)^{\frac{\theta}{2\pi}}(b-x)^{1-\frac{\theta}{2\pi}}\\
	&+\lim_{x\rightarrow b^-} \frac{B}{\pi(A^2+B^2)}\int_a^b\frac{(t-a)^{\theta/(2\pi)}f(t)(b-t)^{1-(\theta/(2\pi)-\mu)}}{t-x} dt\\
	&=0+0\\
	&=0.
	\end{aligned}
	\end{equation}
	
	This completes the proof of part $(5)$, and the whole proof of Lemma \ref{lemma-Pre-RaiseRegularity} is completed.
\end{proof}
The following lemma provides a bridge connecting the solutions of  the generalized Abel integral equations in the Sobolev space to the solutions  in the space $H^*(\Omega)$. By which we mean that, if a solution $\psi(x)$ of the generalized Abel integral equation is located in  $\widehat{H}^{(1+\mu)/2}_0(\Omega)$ (see equation \eqref{eq:RaisingRegularity}), then its $\mu$-th order derivative ${_aD_x^\mu}\psi$ actually has a representative belonging to $H^*(\Omega)$ (which is equivalent to what equation \eqref{equ-Representation of J-1} says).
This is an important connection and preparation for us to  continue to raise the regularity of ${_aD_x^\mu}\psi$ (namely $J(x)$) to better spaces $H^*_\sigma(\Omega)$ from $H^*(\Omega)$ in  our main Theorem~\ref{lem: RaisingTheRegularity}.

\begin{lemma}\label{lem-raising ragularity-1}
	Let $c$ be a constant, $\psi(x)\in \widehat{H}^{(1+\mu)/2}_0(\Omega)$ and $f(x)\in H(\overline{\Omega}) $. If 
	\begin{equation}\label{eq:RaisingRegularity}
	\alpha _aD_x^{-(1-\mu)}\psi+\beta {_xD_b^{-(1-\mu)}}\psi \overset{a.e.}{=}{_aD_x^{-1}}f+c, \; x\in \Omega, 
	\end{equation}
	then the solution $\psi(x)$ has a representation
	\begin{equation}\label{equ-Representation of J-1}
	\psi(x)={_aD_x^{-\mu}}J,
	\end{equation}
	where $J(x)\in H^*(\Omega)$ and $J(x)$ has four equivalent representations:
	\begin{equation}\label{equ-Representation of J-2}
	\begin{aligned}
	J(x)&=\frac{A f(x)}{A^2+B^2}-\frac{B}{\pi(A^2+B^2)}\int_a^b\left( \frac{x-a}{t-a}\right)^{1-\frac{\theta}{2\pi}}\left(\frac{b-x}{b-t}\right)^{\frac{\theta}{2\pi}-\mu} \frac{ f(t)}{t-x} dt,\\
	J(x)&=\frac{A f(x)}{A^2+B^2}-\frac{B}{\pi(A^2+B^2)}\int_a^b\left( \frac{x-a}{t-a}\right)^{1-\frac{\theta}{2\pi}}\left(\frac{b-x}{b-t}\right)^{\frac{\theta}{2\pi}-\mu-1} \frac{ f(t)}{t-x} dt,\\
	J(x)&=\frac{A f(x)}{A^2+B^2}-\frac{B}{\pi(A^2+B^2)}\int_a^b\left( \frac{x-a}{t-a}\right)^{-\frac{\theta}{2\pi}}\left(\frac{b-x}{b-t}\right)^{\frac{\theta}{2\pi}-\mu} \frac{ f(t)}{t-x} dt,\\
	J(x)&=\frac{A f(x)}{A^2+B^2}-\frac{B}{\pi(A^2+B^2)}\int_a^b\left( \frac{x-a}{t-a}\right)^{-\frac{\theta}{2\pi}}\left(\frac{b-x}{b-t}\right)^{\frac{\theta}{2\pi}-\mu-1} \frac{ f(t)}{t-x} dt.
	\end{aligned}
	\end{equation}
	where 	$A=\alpha-\beta\cos(\mu\pi)$, $B=\beta\sin(\mu\pi)$.
\end{lemma}

\begin{proof}\;
	1.  Differentiating both sides of equation~\eqref{eq:RaisingRegularity} is valid by the assumption $\psi\in \widehat{H}_0^{(1+\mu)/2}(\Omega)$ and Property~\ref{pro:alternate-form}, from which we have
	\begin{equation}
	D(\alpha _aD_x^{-(1-\mu)}\psi+\beta {_xD_b^{-(1-\mu)}}\psi )\overset{a.e.}{=}f.
	\end{equation}
	Distributing the differentiation operator $D$ is permitted and gives
	\begin{equation}\label{equ:Raising-1}
	\alpha {_aD_x^\mu}\psi-\beta {_xD_b^\mu}\psi\overset{a.e.}{=}f.
	\end{equation}
	
	Noting $(1+\mu)/2>\mu$ and recalling the knowledge of embedding $\widehat{H}^{(1+\mu)/2}_0(\Omega)$   $\subset$   $ \widehat{H}^{\mu}_0(\Omega)$, we see  from Property~\ref{pro:alternate-form}  that  $\psi(x)$ can be represented as $\psi(x)={_aD_x^{-\mu}}{_aD_x^\mu}\psi$ and ${_aD_x^\mu}\psi\in L^2(\Omega)$.   
	
	Doing back substitution for $\psi$ into \eqref{equ:Raising-1}, the left-hand side becomes
	\begin{equation}\label{equ-SubstitutingLeftIngegral}
	\alpha {_aD_x^\mu}\psi-\beta {_xD_b^\mu}\psi=\alpha {_aD_x^\mu}\psi-\beta {_xD_b^\mu}{_aD_x^{-\mu}}{_aD_x^\mu}\psi.
	\end{equation}
	
	For notation simplicity, we denote $r_a(x)=x-a, r_b(x)=b-x,  x\in \overline{\Omega}$.
	
	Applying the fact (eq. (11.17), p. 206, \cite{MR1347689}) that 
	\[
	{_aD_x^{-\mu}} g= \cos(\mu \pi){_xD_b^{-\mu}}g-\sin(\mu \pi) {_xD_b^{-\mu}}(r_b^{-\mu}S(r_b^{\mu}g)),\,\forall g(x)\in L^p(\Omega),p> 1
	\]
	to the right-hand side of \eqref{equ-SubstitutingLeftIngegral}, we have
	\begin{equation}
	\begin{aligned}
	&=\alpha {_aD_x^\mu}\psi-\beta \cos(\mu\pi){_aD_x^\mu}\psi+\beta\sin(\mu\pi)r_b^{-\mu}S(r_b^\mu{_aD_x^\mu}\psi)\\
	&=\left( \alpha-\beta\cos(\mu\pi)\right){_aD_x^\mu}\psi+\beta\sin(\mu\pi)r_b^{-\mu}S(r_b^\mu{_aD_x^\mu}\psi).
	\end{aligned}
	\end{equation}
	Inserting this back into equation~\eqref{equ:Raising-1},  multiplying both sides by $r_b^\mu$ and denoting \[
	A=\alpha-\beta\cos(\mu\pi),\, B=\beta\sin(\mu\pi),\, \Psi_1(x)=r_b^\mu{_aD_x^\mu}\psi,\, F(x)=r_b^\mu f(x), 
	\]
	we arrive at
	\begin{equation}\label{equ:SDI}
	A\Psi_1(x) + \frac{B}{\pi}\int_a^b\frac{\Psi_1(t)}{t-x} dt \overset{a.e.}{=}F(x), x\in \Omega.
	\end{equation}
	
	\vspace{0.2cm}
	
	2.  On the other hand, by Lemma~\ref{lemma-Pre-RaiseRegularity}, we already know that there exist solutions $\Psi_2 \in H^*(\Omega)\cap C((a, b])$ and $\Psi_3\in H^*(\Omega)\cap C([a, b))$ satisfying
	\begin{equation}\label{equ:SDI2} 
	A\Psi_i(x)+ \frac{B}{\pi}\int_a^b\frac{\Psi_i(x)}{t-x} dt =F(x), x\in \Omega, (i=2,3),
	\end{equation}
	and that
	\begin{equation}\label{equ-BelongtoH^*}
	\frac{\Psi_2}{(b-x)^\mu}, \frac{\Psi_3}{(b-x)^\mu} \in H^*(\Omega), \quad \mu<\frac{\theta}{2\pi}<1.
	\end{equation}
	
	Notice that the distinction between \eqref{equ:SDI} and \eqref{equ:SDI2} is that \eqref{equ:SDI} holds $a.e.$ and \eqref{equ:SDI2} holds for every $x\in \Omega$ and that $\Psi_2$ and $\Psi_3$ belong to $H^*(\Omega)$ but $\Psi_1$ is not immediately clear yet for now. 
	
	In the following, our strategy is to intend to show that
	\begin{equation}
	\Psi_1(x)=\Psi_2(x)=\Psi_3(x),
	\end{equation}
	which  produces \eqref{equ-Representation of J-1} in the lemma, and after this, \eqref{equ-Representation of J-2} will be obtained shortly, the whole lemma hence will be eventually completed.
	
	\vspace{0.2cm}
	
	3. Now let us continue. 
	
	Subtracting \eqref{equ:SDI2} from \eqref{equ:SDI} gives
	\begin{equation}\label{equ:SDI23}
	A\Psi(x) + \frac{B}{\pi}\int_a^b\frac{\Psi(t)}{t-x} dt\overset{a.e.}{=}0,\, x\in \Omega,
	\end{equation}
	where 
	\[
	\Psi(x)=\Psi_1(x)-\Psi_i(x), (i=2,3).
	\]
	
	(Evidently, $\Psi(x)$ depends on $i$ and it should not be confused. We denote $\Psi(x)$ this way simply because it is more convenient to discuss this two cases $i=2,3$ together in the rest of proof rather than separately.)

	Dividing both sides of \eqref{equ:SDI23} by $r_b^\mu$, we arrive at
	\begin{equation}\label{equ:SDI234}
	A\widetilde{\Psi}(x) + \frac{B}{\pi}\frac{1}{(b-x)^\mu}\int_a^b\frac{(b-t)^\mu \widetilde{\Psi}(t)}{t-x} dt\overset{a.e.}{=}0, \, x\in \Omega,
	\end{equation}
	where 
	\[
	\widetilde{\Psi}(x)=\frac{\Psi(x)}{(b-x)^\mu}=\frac{\Psi_1(x)}{(b-x)^\mu}-\frac{\Psi_i(x)}{(b-x)^\mu}, (i=2, 3).
	\]
	
	\vspace{0.2cm}
	
	4. Let us examine which functional space $\widetilde{\Psi}(x)$ belongs to  (discussing the two cases $i=2,3$ together).
	
	First,
	\[
	\frac{\Psi_1(x)}{(b-x)^\mu}={_aD_x^\mu}\psi \in L^2(\Omega), \, \text{since $\psi\in \widehat{H}^{(1+\mu)/2}_0(\Omega)$}.
	\]
	Secondly, since 
	\begin{equation}
	\frac{\Psi_i}{(b-x)^\mu}\in H^*(\Omega),
	\end{equation}
	we can always find a certain $p>1$ that does not depend on $i$ such that
	\begin{equation}
	\frac{\Psi_i}{(b-x)^\mu}\in L^p(\Omega), (i=2,3),
	\end{equation}
	by taking into account the definition of  $H^*(\Omega)$.
	
	Combining these together, we  conclude that
	\begin{equation}\label{eq-find a space}
	\widetilde{\Psi}(x) \in L^p(\Omega),\, \text{for a certain $p>1$}.
	\end{equation}
	
	\vspace{0.2cm}
	
	5. The establishment of  \eqref{eq-find a space}  allows us to be  able to  apply Property~\ref{pro:CR} to equation~\eqref{equ:SDI234}. To do so, integrating both sides of ~\eqref{equ:SDI234} by ${_xD_b^{-\mu}}$, we obtain
	\begin{equation}
	A{_xD_b^{-\mu}}\widetilde{\Psi} + \frac{B}{\pi}(x-a)^{\mu}\int_a^b\frac{{_tD_b^{-\mu}}\widetilde{\Psi} }{(t-a)^\mu(t-x)} dt\overset{a.e.}{=}0, \, x\in \Omega.
	\end{equation}
	Dividing both sides by $(x-a)^\mu$, we arrive at
	\begin{equation}\label{equ:SDI2345}
	A\widetilde{\widetilde{\Psi}}(x) + \frac{B}{\pi}\int_a^b\frac{\widetilde{\widetilde{\Psi}}(t)}{t-x} dt\overset{a.e.}{=}0, \, x\in \Omega,
	\end{equation}
	where
	\[
	\widetilde{\widetilde{\Psi}}(x)=\frac{{_xD_b^{-\mu}}\widetilde{\Psi}}{(x-a)^\mu}.
	\]
	
	Before going further, let us call attention to that \eqref{equ:SDI2345} holds $a.e.$ at this stage since $\widetilde{\widetilde{\Psi}}(x)$ is essentially  in terms of $\psi(x)$, and in the next two steps we intend to show that $\widetilde{\widetilde{\Psi}}(x)$ actually admits a good representative such that \eqref{equ:SDI2345} holds for every $x\in \Omega$.
	
	\vspace{0.2cm}
	
	6. Now we assert that $\widetilde{\widetilde{\Psi}}(x)\in H^*(\Omega)$ (for both $i=2,3$). (It should be clear that $\widetilde{\widetilde{\Psi}}(x)\in H^*(\Omega)$ means there is a representative in the equivalence classes of $\widetilde{\widetilde{\Psi}}(x)$ such that it belongs to $ H^*(\Omega)$).

	To see this, substituting for $\widetilde{\Psi}$ into $\widetilde{\widetilde{\Psi}}(x)$, we  have
	\begin{equation}\label{equ-Doulble tidle}
	\begin{aligned}
	\widetilde{\widetilde{\Psi}}(x)&=\frac{1}{(x-a)^\mu} {_xD_b^{-\mu}}\left(\frac{1}{r_b^\mu}\Psi_1-\frac{1}{r_b^\mu}\Psi_i \right)\\
	&=\frac{1}{(x-a)^\mu}{_xD_b^{-\mu}}{_aD_x^\mu}\psi -\frac{1}{(x-a)^\mu}{_xD_b^{-\mu}}\frac{\Psi_i}{r_b^\mu}\\
	&=M_1-M_2.
	\end{aligned}
	\end{equation}
	
	It suffices to investigate $M_1$ and $M_2$, respectively.
	
	Consider $M_1$ first and examine  the piece ${_xD_b^{-\mu}}{_aD_x^\mu}\psi$. By Property~\ref{pro:alternate-form}, there exists a function $\psi_1(x)\in L^2(\Omega)$ such that 
	\[
	{_aD_x^\mu}\psi={_aD_x^\mu}{_aD_x^{-(1+\mu)/2}}\psi_1={_aD_x^{-(1-\mu)/2}}\psi_1.
	\]
	On the other hand,   there exists a  function $\psi_2(x)\in L^2(\Omega)$ such that
	\[
	{_aD_x^{-(1-\mu)/2}}\psi_1={_xD_b^{-(1-\mu)/2}}\psi_2,
	\]
	(Corollary 1, p. 208, \cite{MR1347689}). 
	Thus,
	\[
	{_xD_b^{-\mu}}{_aD_x^\mu}\psi={_xD_b^{-(1+\mu)/2}}\psi_2\in  H^{\mu/2}  (\overline{\Omega})
	\]
	is guaranteed by Property \ref{pro:MappingIntoHolderianFromLp}.
	
	If we equivalently rewrite $M_1$ as
	\begin{equation}\label{eq-Rerepresentation}
	M_1=\frac{(x-a)^\epsilon ({_xD_b^{-\mu}}{_aD_x^\mu}\psi) (b-x)^\epsilon}{(x-a)^{1-(1-\mu-\epsilon)}(b-x)^{1-(1-\epsilon)}},
	\end{equation}
	where $\epsilon$ is chosen so that $0<\epsilon<\min\{1-\mu,\mu/2\}$, then by simple steps as we justified for $\Pi_1$ in the step 3 of  the proof of Lemma~\ref{lemma-Pre-RaiseRegularity},
	we see
	\[
	M_1\in H^*(\Omega),
	\]
	and not repeated here.
	
	For  $M_2$, indeed, by virtue of Property~\ref{pro-correspondence-1} and recalling \eqref{equ-BelongtoH^*}, ${_xD_b^{-\mu}}\frac{\Psi_i}{r_b^\mu}$ can be represented as
	\begin{equation}
	{_xD_b^{-\mu}}\frac{\Psi_i}{r_b^\mu}=\frac{g_i(x)}{(x-a)^{1-\epsilon_1}(b-x)^{1-\epsilon_2}},
	\end{equation}
	for certain functions $g_i(x)$ and real numbers $k,\epsilon_1,\epsilon_2$ satisfying
	\begin{equation}
	g_i(x)\in H^k_0(\overline{\Omega}),\quad \mu<k, 0<\epsilon_1, 0<\epsilon_2 \quad (\text{$k,\epsilon_1,\epsilon_2$ depend on $i$}).
	\end{equation}
	Taking into account the useful fact that
	\begin{equation}
	\frac{g_i(x)}{(x-a)^\mu}\in H^l_0(\overline{\Omega}) , \quad \text{for any $0<l<k-\mu$},
	\end{equation}
	(the value of $\frac{g_i(x)}{(x-a)^\mu}$ at $x=a$ is understood in the  limiting sense) and the definition of  $H^*(\Omega)$, we see
	
	\begin{equation}
	M_2=\frac{1}{(x-a)^\mu}{_xD_b^{-\mu}}\frac{\Psi_i}{r_b^\mu}=\frac{\frac{g_i(x)}{(x-a)^\mu} }{(x-a)^{1-\epsilon_1}(b-x)^{1-\epsilon_2}}\in H^*(\Omega), (i=2,3).
	\end{equation}
	
	Combing $M_1$ and $M_2$ yields the assertion
	\[
	\widetilde{\widetilde{\Psi}}(x)\in H^*(\Omega)\quad  \text{for both}\, i=2,3.
	\]
	
	\vspace{0.2cm}
	
	7. Once we have the above, it now is notable that equation~\eqref{equ:SDI2345} becomes valid for every point in $\Omega$, namely
	\begin{equation}\label{lemma2dominantSI}
	A\widetilde{\widetilde{\Psi}}(x) + \frac{B}{\pi}\int_a^b\frac{\widetilde{\widetilde{\Psi}}(t)}{t-x} dt=0,\, \text{for each $x\in \Omega$}.
	\end{equation}
	
	(Putting it another way, there is a representative for the equivalence classes of $\widetilde{\widetilde{\Psi}}(x) $ such that it belongs to $H^*(\Omega)$ and makes \eqref{equ:SDI2345} hold for every $x\in \Omega$ ).
	
	On the other hand, remember that, in $H^*(\Omega)$, \eqref{lemma2dominantSI} is unconditionally solvable according to the part $(1)$ of Property~\ref{pro:SolvabilityofDSI}. Hence by utilizing \eqref{eq-sloution-CSE}, we derive that  $\widetilde{\widetilde{\Psi}}(x)$ can be represented as a constant multiple of $(x-a)^{1-n_a-\frac{\theta}{2\pi}}(b-x)^{\frac{\theta}{2\pi}-n_b}$ with choosing $n_a=1,n_b=1$, namely
	\begin{equation}\label{equ:SI}
	\widetilde{\widetilde{\Psi}}(x)=C_i \cdot(x-a)^{-\frac{\theta}{2\pi}}(b-x)^{\frac{\theta}{2\pi}-1}, 
	\end{equation}
	where $C_i$ depends on the $i$ in $\widetilde{\widetilde{\Psi}}(x)$  (i=2,3).

	In the next two steps, we will show $C_2$ and $C_3$ have to be zero, separately.
	
	\vspace{0.2cm}
	
	8. Using the second line of expression~\eqref{equ-Doulble tidle}, we solve equation~\eqref{equ:SI} for ${_aD_x^\mu}\psi$ to obtain
	\begin{equation}\label{eq-Final Equation}
	{_aD_x^\mu}\psi=\frac{\Psi_i(x)}{(b-x)^\mu}+C_i \,{_xD_b^\mu}((x-a)^{-\frac{\theta}{2\pi}+\mu}(b-x)^{\frac{\theta}{2\pi}-1}),\, (i=2,3).
	\end{equation}

	Integrating both sides by ${_aD_x^{-\mu}}$, which is valid, and also noting $\psi(x)={_aD_x^{-\mu}}{_aD_x^{\mu}}\psi$, we have
	\begin{equation}\label{equ:Psifunctiontosimplify}
	\psi(x)={_aD_x^{-\mu}}\frac{\Psi_i}{r_b^\mu}+C_i \,{_aD_x^{-\mu}}{_xD_b^\mu}((x-a)^{-\frac{\theta}{2\pi}+\mu}(b-x)^{\frac{\theta}{2\pi}-1}),\, (i=2,3).
	\end{equation}
	
	Calculating the second term on the right-hand side (using eq. (11.4) and (11.19), \cite{MR1347689}),
	\begin{equation}
	\begin{aligned}
	&{_aD_x^{-\mu}}{_xD_b^\mu}((x-a)^{-\frac{\theta}{2\pi}+\mu}(b-x)^{\frac{\theta}{2\pi}-1})\\
	&=\left(\cos(\mu \pi)+\sin(\mu \pi)\cot(\pi-\frac{\theta}{2}) \right)(x-a)^{-\frac{\theta}{2\pi}+\mu}(b-x)^{\frac{\theta}{2\pi}-1}\\
	&=\widetilde{C} \cdot (x-a)^{-\frac{\theta}{2\pi}+\mu}(b-x)^{\frac{\theta}{2\pi}-1},
	\end{aligned}
	\end{equation}
	where
	\[
	\widetilde{C} =\cos(\mu \pi)+\sin(\mu \pi)\cot(\pi-\frac{\theta}{2}).
	\]
	Equation \eqref{equ:Psifunctiontosimplify} further becomes
	\begin{equation}\label{constantshavetobezero}
	\psi(x)={_aD_x^{-\mu}}\frac{\Psi_i}{r_b^\mu}+C_i\widetilde{C} \cdot (x-a)^{-\frac{\theta}{2\pi}+\mu}(b-x)^{\frac{\theta}{2\pi}-1}, \, (i=2,3).
	\end{equation}
	
	Consider the case $i=3$ now. On one hand, recall that $\mu<\frac{\theta}{2\pi}<1$ from \eqref{equ-BelongtoH^*}. This implies that the coefficient $\widetilde{C}$ is non-zero and hence that $\widetilde{C} \cdot(x-a)^{-\frac{\theta}{2\pi}+\mu}(b-x)^{\frac{\theta}{2\pi}-1}$ is unbounded at the boundary point $x=a$. On the other hand,   both $\psi(x)$ (belonging to $\widehat{H}^{(1+\mu)/2}_0(\Omega)$) and ${_aD_x^{-\mu}}\frac{\Psi_3}{r_b^\mu}$  (remember $\Psi_3\in H^*(\Omega)\cap C([a, b))$ ) are bounded at $x=a$, which contradicts the unboundedness of $C_3\widetilde{C} \cdot (x-a)^{-\frac{\theta}{2\pi}+\mu}(b-x)^{\frac{\theta}{2\pi}-1}$ unless $C_3=0$. Thereby, we  arrive at
	\begin{equation}\label{equ-first-equality-for-i=3}
	\psi(x)={_aD_x^{-\mu}}\frac{\Psi_3}{r_b^\mu}.
	\end{equation}
	
	So, we have proved \eqref{equ-Representation of J-1} in the lemma by letting $J(x)=\frac{\Psi_3}{r_b^\mu}$ and noting $J(x) \in H^*(\Omega)$ from \eqref{equ-BelongtoH^*}. We remain to show that $\frac{\Psi_3}{r_b^\mu}$ has four equivalent representations, namely \eqref{equ-Representation of J-2}, which follows from the next step.
	
	\vspace{0.2cm}
	
	9.   In this last step we  show that the constant $C_2$ in \eqref{constantshavetobezero} for the case $i=2$ has to be zero as well, namely, $C_2=0$.
	
	Using equation~\eqref{constantshavetobezero} and doing subtraction with each other for $i=2,3$, 
	\begin{equation}\label{equ:DoingSubtracting}
	{_aD_x^{-\mu}}\frac{\Psi_3}{r_b^\mu}-{_aD_x^{-\mu}}\frac{\Psi_2}{r_b^\mu}=C_2\widetilde{C}\cdot(x-a)^{-\frac{\theta}{2\pi}+\mu}(b-x)^{\frac{\theta}{2\pi}-1}.
	\end{equation}

	Let us take care about each term  above  at the  boundary point $x=b$ to  obtain a contradiction.
	
	For the left-hand side, using  \eqref{equ-first-equality-for-i=3} and invoking the part $(4)$ of Lemma~\ref{lemma-Pre-RaiseRegularity}, we know
	\begin{equation}
	{_aD_x^{-\mu}}\frac{\Psi_3}{r_b^\mu} \left(=\psi(x)\in \widehat{H}^{(1+\mu)/2}_0(\Omega)\right), \,\, {_aD_x^{-\mu}}\frac{\Psi_2}{r_b^\mu}  \in L^{p}(I_b),
	\end{equation}
	where $p=\frac{1}{1-\frac{\theta}{2\pi}},\,I_b=(\frac{b-a}{2}, b) $. However, in the right-hand side of \eqref{equ:DoingSubtracting}, it is clear that
	\begin{equation}
	\widetilde{C}\cdot(x-a)^{-\frac{\theta}{2\pi}+\mu}(b-x)^{\frac{\theta}{2\pi}-1}  \notin L^{p}(I_b).
	\end{equation}

	This means that $C_2$ has to be zero in \eqref{equ:DoingSubtracting}, from which it follows that
	\begin{equation}\label{equ-first-equality-for-i=2}
	\psi(x)={_aD_x^{-\mu}}\frac{\Psi_2}{r_b^\mu}.
	\end{equation}
	
	Comparing \eqref{equ-first-equality-for-i=3} and \eqref{equ-first-equality-for-i=2},  we see 
	\begin{equation}
	{_aD_x^{-\mu}}(\frac{\Psi_3}{r_b^\mu}-\frac{\Psi_2}{r_b^\mu})=0.
	\end{equation}
	
	Only trivial solution is allowed by  seeing $\frac{\Psi_2}{r_b^\mu}, \frac{\Psi_3}{r_b^\mu} \in H^*(\Omega)$ from \eqref{equ-BelongtoH^*}, and thus
	\begin{equation}\label{equ:equality}
	\Psi_2(x)=\Psi_3(x),\, x\in \Omega.
	\end{equation}
	
	Once we have  equality \eqref{equ:equality},  by virtue of the part (5) of Lemma~\ref{lemma-Pre-RaiseRegularity}, \eqref{equ-Representation of J-2}, namely the four desired representations for $\frac{\Psi_3}{r_b^\mu}$ (or $\frac{\Psi_2}{r_b^\mu}$) in the lemma,  follows immediately  after dividing by $r_b^\mu$.
	
	This finally completes the whole proof.
\end{proof}
Now we go one  step further from above lemma by showing that ${_aD_x^\mu}\psi$  can actually go  to better spaces $H_\sigma^*(\Omega)$ from  $H^*(\Omega)$ provided that $f$ lies in $H_\sigma^*(\Omega)$ (which is the same as what \eqref{equ:lemma3representation} means in the following).

\begin{theorem}\label{lem: RaisingTheRegularity}
	Given  $0<\sigma<1$,  let $c$ be a constant, $\psi(x)\in \widehat{H}^{(1+\mu)/2}_0(\Omega)$ and $f(x)\in H(\overline{\Omega}) \cap H^*_\sigma(\Omega)$. If 
	\begin{equation}\label{equ-GoFurther}
	\alpha _aD_x^{-(1-\mu)}\psi+\beta {_xD_b^{-(1-\mu)}}\psi \overset{a.e.}{=}{_aD_x^{-1}}f +c, x\in\Omega,
	\end{equation}
	then the solution $\psi(x)$ has a representation
	\begin{equation}\label{equ:lemma3representation}
	\psi(x)={_aD_x^{-(\mu+\sigma)}}K_\sigma,
	\end{equation}
	where function $K_\sigma(x)$ belongs to $H^*(\Omega)$ ($K_\sigma$ depends on $\sigma$).
\end{theorem}	
\begin{proof}  Notice that, compared to  Lemma~\ref{lem-raising ragularity-1}, only one more condition is imposed in this lemma, namely $f(x)\in H^*_\sigma(\Omega), 0<\sigma<1$. 
	
	\vspace{0.2cm}
	
	1. As a direct consequence of Lemma~\ref{lem-raising ragularity-1}, we  already know that the function $\psi(x)$ must have a representation
	\begin{equation}\label{equ-FirstIntegralEx}
	\psi(x)={_aD_x^{-\mu}}J,
	\end{equation}
	where $J(x)\in H^*(\Omega)$ and $J(x)$ has four equivalent representations:
	\begin{equation}\label{equ-FourDiffRepresentations}
	\begin{aligned}
	J(x)&=\frac{A f(x)}{A^2+B^2}-\frac{B}{\pi(A^2+B^2)}\int_a^b\left( \frac{x-a}{t-a}\right)^{1-\frac{\theta}{2\pi}}\left(\frac{b-x}{b-t}\right)^{\frac{\theta}{2\pi}-\mu} \frac{ f(t)}{t-x} dt;\\
	J(x)&=\frac{A f(x)}{A^2+B^2}-\frac{B}{\pi(A^2+B^2)}\int_a^b\left( \frac{x-a}{t-a}\right)^{1-\frac{\theta}{2\pi}}\left(\frac{b-x}{b-t}\right)^{\frac{\theta}{2\pi}-\mu-1} \frac{ f(t)}{t-x} dt;\\
	J(x)&=\frac{A f(x)}{A^2+B^2}-\frac{B}{\pi(A^2+B^2)}\int_a^b\left( \frac{x-a}{t-a}\right)^{-\frac{\theta}{2\pi}}\left(\frac{b-x}{b-t}\right)^{\frac{\theta}{2\pi}-\mu} \frac{ f(t)}{t-x} dt;\\
	J(x)&=\frac{A f(x)}{A^2+B^2}-\frac{B}{\pi(A^2+B^2)}\int_a^b\left( \frac{x-a}{t-a}\right)^{-\frac{\theta}{2\pi}}\left(\frac{b-x}{b-t}\right)^{\frac{\theta}{2\pi}-\mu-1} \frac{ f(t)}{t-x} dt.
	\end{aligned}
	\end{equation}
	where 	$A=\alpha-\beta\cos(\mu\pi)$, $B=\beta\sin(\mu\pi)$.
	
	For ease of notation, we denote \eqref{equ-FourDiffRepresentations} as 
	\begin{equation}
	\begin{aligned}
	J(x)&=\text{Expression}\&1;\\
	J(x)&=\text{Expression}\&2;\\
	J(x)&=\text{Expression}\&3;\\
	J(x)&=\text{Expression}\&4.
	\end{aligned}
	\end{equation}
	
	2. For any given $0<\sigma<1$, $\sigma$  must satisfy one of the following inequality cases:
	\begin{equation}
	\begin{aligned}
	\text{Case}\&1&: \sigma+\frac{\theta}{2\pi}\geq 1, \sigma +\mu\geq \frac{\theta}{2\pi};\\
	\text{Case}\&2&: \sigma+\frac{\theta}{2\pi}\geq 1, \sigma +\mu< \frac{\theta}{2\pi};\\
	\text{Case}\&3&: \sigma+\frac{\theta}{2\pi}< 1, \sigma +\mu\geq \frac{\theta}{2\pi};\\
	\text{Case}\&4&: \sigma+\frac{\theta}{2\pi}< 1, \sigma +\mu<\frac{\theta}{2\pi}.\\
	\end{aligned}
	\end{equation}
	
	3. We associate different cases to different expressions as follows:
	\begin{equation}\label{equ-MixedTogether}
	\begin{aligned}
	\text{Case}\&1& \quad \text{with} \quad\text{Expression}\&1;\\
	\text{Case}\&2&\quad\text{with} \quad \text{Expression}\&2;\\
	\text{Case}\&3&\quad\text{with} \quad\text{Expression}\&3;\\
	\text{Case}\&4&\quad\text{with} \quad \text{Expression}\&4.\\
	\end{aligned}
	\end{equation}
	
	For each of \eqref{equ-MixedTogether}, applying Property~\ref{pro-WeightedMapping} to the according expression of $J(x)$ is valid by checking the two inequality conditions in Property~\ref{pro-WeightedMapping} and  yields that
	\begin{equation}
	J(x)\in H^*_\sigma(\Omega).
	\end{equation}

	4. Utilizing Property~\ref{pro-correspondence-1}, we know that there exists a  certain function $K_\sigma(x)\in H^*(\Omega)$ such that 
	\begin{equation}\label{equ:Tobeinserted}
	J(x)={_aD_x^{-\sigma}}K_\sigma,
	\end{equation}
	and therefore, by inserting \eqref{equ:Tobeinserted} back into \eqref{equ-FirstIntegralEx},
	\begin{equation}
	\psi(x)={_aD_x^{-(\mu+\sigma)}}K_\sigma
	\end{equation}
	follows from the semigroup property of R-L integral operators. 
	
	This completes the whole proof.
\end{proof}


\bibliography{mybibfile}

\end{document}